\documentclass[12pt,a4paper]{article}
\usepackage[english]{babel}
\usepackage{amsmath,amssymb,graphicx}
\usepackage{comment}
\usepackage{fullpage} %style
\usepackage{hyperref}
\usepackage{mathtools}
\numberwithin{equation}{section}
\usepackage[utf8]{inputenc}
\usepackage{amsfonts}
\usepackage{amsthm}
\usepackage{xcolor}
\usepackage{enumerate}

\def\cF{\mathcal{F}}
\def\cP{\mathcal{P}}

\def\bR{\mathbb{R}}
\def\bT{\mathbb{T}}

\def\bE{\mathbb{E}}
\def\bP{\mathbb{P}} 

\def\kin{\mathtt{k}}

\newcommand{\norm}[1]{\left\lVert#1\right\rVert}

\newcommand{\dd}{\textup{d}}

\theoremstyle{definition}
\newtheorem{definition}{Definition}[section]

\theoremstyle{plain}
\newtheorem{theorem}[definition]{Theorem}
\newtheorem{proposition}[definition]{Proposition}
\newtheorem{lemma}[definition]{Lemma}

\theoremstyle{remark}
\newtheorem{remark}[definition]{Remark}

\title{A law of large numbers\\ for kinetic interacting diffusions}

\date{}

\author{
 Carlo Bellingeri  \thanks{ \texttt{carlo.bellingeri@uha.fr}  IRIMAS, UHA, 18 Rue des Frères Lumières, 68200 Mulhouse, France}
   \and
Fabio Coppini \thanks{\texttt{f.coppini@uu.nl} Utrecht University, Budapestlaan 6, 3584 CD Utrecht, The Netherlands}
}

\begin{document}
\maketitle
\begin{abstract}
We study the convergence of the empirical distribution associated with a system of interacting kinetic particles subject to independent Brownian forcing in a finite horizon setting, using some recent progress on kinetic non-linear partial differential equations. Under general assumptions that require only weak convergence on the initial datum -without assuming independence or moment conditions- we prove convergence in probability to the corresponding non-linear Fokker-Planck PDE.
\end{abstract}
 \medskip
 

\medskip

%\tableofcontents

\section{Introduction}
Fix $d\geq 1$ and $T>0$, we start with a particle system of $N=2, 3, \dots$ particles, where each particle, labelled $i = 1,\dots , N$, is described by a position $x^{i,N}\in \bR^d$ and a velocity $v^{i,N}\in \bR^d$. The evolution is described by a system of stochastic differential equations over $ \bR^{2d}=\bR^d\times\bR^d$ for $0\leq t\leq T$
\begin{equation}
    \label{eq:kinetic-diffusions}
 \left\{ \; \,\begin{aligned}
    &\dd x^{i,N}_t = v^{i,N}_t \dd t,\\&
    \dd v^{i,N}_t = \left(\frac{1}{N}\sum_{j\neq i} \Gamma((x^{i,N}_t,v^{i,N}_t),(x^{j,N}_t,v^{j,N}_t))\right)\dd t + \sigma \dd B^i_t
     \end{aligned}\right.\,
\end{equation}
where $\Gamma: \bR^{2d}\times \bR^{2d} \to \bR^d$ is a sufficiently regular function describing the interaction among particles, $\sigma \geq 0$ and $(B^{1, N}_t, \dots, B^{N,N}_t)_{t\geq 0}$ is a vector of independent and identically distributed (IID) Brownian motions on $\bR^d$ defined on a suitable probability space $(\Omega, \mathcal{F},\bP)$. In order to understand a large population of particles, it is helpful to focus on their distribution, i.e., the empirical measure
\begin{equation}
    \label{def:emp-measure-kinetic-diffusions}
    \nu^N_t = \frac 1N \sum_{j=1}^N\delta_{(x^{i,N}_t, v^{i,N}_t)}, \qquad \text{for } 0\leq t\leq T,
\end{equation}
a random variable with values in $\mathcal{P}(\bR^{2d})$, the probability measures on $\bR^{2d}$. 
In the deterministic case when $\sigma=0$, the evolution over time of $\nu^N := (\nu^N_t)_{t\in[0,T]}$ is given by the Liouville equation for the density of one particle
\begin{equation}
\label{def:vlasov-equation}
    \partial_t\nu_t +v\cdot \nabla_x\nu_t = -\text{div}_v\left[\nu_t (\Gamma*\nu_t)\right], \quad 0 \leq t \leq T,
\end{equation}
where $*$ denotes the convolution with respect to the second argument, i.e., for $\mu \in \cP(\bR^{2d})$, \[(\Gamma*\mu)(x,v):=\int_{\bR^{2d}}\Gamma((x,v), (y,w)) \, \mu( \dd y, \dd w) \quad \text{for $(x,v) \in \bR^{2d}$\,.}\]
Indeed, for any test function $f\colon \bR^{2d}\to \bR$ and $1\leq i \leq N$, one has that
\begin{equation*}
\begin{split}
    & f(x^{i,N}_t, v^{i,N}_t) - f(x^{i,N}_0, v^{i,N}_0) = \\
    &=\int_0^t \nabla_x f(x^{i,N}_s, v^{i,N}_s) \cdot v^{i,N}_s \dd s + \int_0^t \nabla_v f(x^{i,N}_s, v^{i,N}_s) \cdot (\Gamma *\nu^N_s)(x^{i,N}_s, v^{i,N}_s) \dd s,
\end{split}
\end{equation*}
and by taking $N^{-1}\sum_{i=1}^N$ on both sides, one obtains the weak formulation of Equation \eqref{def:vlasov-equation}. In the beginning of the last century, Equation \eqref{def:vlasov-equation} appeared with the specific form $\Gamma((x,v),(y,w) )=-\nabla_x V(y-x) $ in the modelling of the dynamics of galaxies \cite{jeans_theory_1915, jeans_theory_1916} and, only fifty years later, was employed by Vlasov \cite{vlasov_vibrational_1968} using a  more general $\Gamma$ to study a variant of Boltzmann equation, see, e.g., the introduction in \cite{dietert_contributions_2016} and references therein for a historical perspective. In the second half of the 70's, Neunzert \cite{neunzert_neuere_1975,neunzert_introduction_1984}, rapidly followed by the works of Braun and Hepp \cite{braun_vlasov_1977} and Dobrushin \cite{dobrushin_vlasov_1979}, proved that, if $\nu^N_0$ weakly converges to $\nu_0\in\mathcal{P}(\bR^{2d})$, then $\nu^N$ weakly converges\footnote{It should be noted that, although very similar, these works differ in how the space of probability measures is metricized: Neunzert used the bounded Lipschitz distance, while Braun, Hepp and Dobrushin the Wasserstein distance, see \cite{neunzert_introduction_1984}.} in $C([0,T], \mathcal{P}(\bR^{2d}))$ to $\nu=(\nu_t)_{t\in [0,T]}$ solution to Equation \eqref{def:vlasov-equation}.

In the stochastic setting, when $\sigma > 0$, the empirical measure is random and cannot satisfy a deterministic partial differential equation. However, in the limit for $N\to\infty$ and under additional assumptions on the initial condition, one can show \cite{mckean_class_1966, oelschlager_martingale_1984, tanaka_limit_1984, sznitman_topics_1991} (among others results) that it converges to the solution of the deterministic PDE
\begin{equation}
\label{kinetic-fokker-planck}
\partial_t\nu_t+ v \cdot\nabla_x \nu_t =\frac {\sigma^2}2 \Delta_v\nu_t - \text{div}_v(\nu_t (\Gamma*\nu_t)), \quad 0\leq t \leq T,
\end{equation}
provided that $\nu^N_0$ converges to $\nu_0$ in a suitable topology. Equation \eqref{kinetic-fokker-planck} is usually known as kinetic non-linear Fokker-Planck equation or kinetic McKean-Vlasov equation and, for $\sigma=0$, it is Equation \eqref{def:vlasov-equation}.

\subsection{Motivation, aim of this work and contributions}

Even though interacting particle systems have been repeatedly studied in the last century, they continue to be adopted for modelling complex systems with applications that vary from biological systems \cite{bertini_synchronization_2014, coghi_pathwise_2020, oliveira_interacting_2019} to economy \cite{parise_graphon_2020, delarue_master_2019} as well as other fields. A crucial aspect for applications is related to the initial conditions that cannot always be assumed independent from each others or concentrated in space, e.g., with finite $p$-moment for $p\geq 1$. To the authors' knowledge, this last aspect has not been completely solved when the dynamics is perturbed by Brownian motions, the existing proofs of the Law of Large Numbers for $\nu^N$ requiring either IID initial condition $(x^{i,N}_0, v^{i,N}_0)_{i=1, \dots,N}$ \cite{sznitman_topics_1991, jabin_quantitative_2018, delarue_master_2019}, or some moment conditions on the large $N$ limit $\nu_0$, e.g., \cite{oelschlager_martingale_1984, leonard_loi_1986, coghi_pathwise_2020} or both in the case of Central Limit Theorems, see, e.g., \cite{fernandez_hilbertian_1997, delarue_master_2019}.
From an applied viewpoint, both assumptions are unsatisfactory; for instance, assuming exchangeable initial conditions can lead to unrealistic models, see, e.g., the discussions in \cite{delattre_note_2016, coppini_note_2022}. Moreover, the hypothesis on the initial data are often mostly technical (in order to handle the sequence of random probability measures $(\nu^N)_N$) and do not necessarily translate physical restrictions: in the deterministic setting, the LLN is a straightforward consequence of the continuity of Equation \eqref{def:vlasov-equation} with respect to the initial conditions, plus the fact that the empirical measure is a solution to \eqref{def:vlasov-equation}, no further assumption on $(x^{i,N}_0, v^{i,N}_0)$ or $\nu_0$ is required. Finally, existence and uniqueness of solutions to \eqref{kinetic-fokker-planck} is assured for any initial probability measure as proved, e.g., by Sznitman.

\begin{proposition}[{\cite{sznitman_topics_1991}}]
\label{pro:known-weak-solution}
 Suppose that $\nu_0\in \mathcal{P}(\bR^{2d})$ and $\Gamma$ is bounded and Lipschitz. Then, there exists a unique  weak solution $\nu$ to Equation \eqref{kinetic-fokker-planck}  such that $\nu \in C([0, T], \cP(\bR^{2d}))$ with $\cP(\bR^{2d})$ endowed with the Kantorovitch-Rubinstein metric.
\end{proposition}
\begin{proof}
    The proof follows  by proving that for the Mckean-Vlasov stochastic differential equation
 \[
 \left\{ \; \,\begin{aligned}
    &\dd X_t = V_t \dd t,\\&
    \dd V_t = \int_{\mathbb{R}^{2d}}\Gamma((X_t,V_t),(y,w))\nu_t^{(X,V)}(\dd y,\dd w)\dd t + \sigma \dd B_t\\&(X_0,V_0)\sim \nu_0\,,\quad  \nu_t^{(X,V)} \;\text{is the law of $(X_t,V_t)$}
     \end{aligned}\right.\,
\]
there is existence and uniqueness of a solution, trajectorial and in law. Therefore we can adapt  the standard argument  in \cite[Theorem 1.1., Lemma 1.3]{sznitman_topics_1991},  as mentioned in \cite[Page 176]{sznitman_topics_1991}.
\end{proof}

The aim of this work is to prove a Law of Large Numbers for the sequence of empirical measures $(\nu^N)_{N=2,3, \dots}$, defined in Equation \eqref{def:emp-measure-kinetic-diffusions}, by assuming comparable hypothesis of the ones in the deterministic setting. It turns out that our proof also follows the main idea of the deterministic setting, showing that the equation satisfied by the empirical measure is close, as $N$ grows, to the kinetic non-linear Fokker-Planck equation satisfied by $\nu$. In the same flavour of the early results of \cite{neunzert_introduction_1984, braun_vlasov_1977, dobrushin_vlasov_1979}, the LLN is a consequence of a control, for each $N$, between the limit solution and the empirical measure evolution. The main result is given by Theorem \ref{thm:lln} which we informally state here.

\begin{theorem}[Informal statement of Theorem \ref{thm:lln}]
\label{thm:informal-lln}
    Let  $\Gamma$ to be a sufficiently regular function.  Then there exists a suitable Hilbert space $(\mathcal{H}_\kin^*, \Vert \cdot\Vert)$ with the following properties: 
    \begin{itemize}
        \item For every $N=2,3, \dots$ $\nu^N = (\nu^N_t)_{t\in[0,T]} \in L^{\infty}([0,T], \mathcal{H}_\kin^*)$ a.s.
        \item Any  weak solution to equation \eqref{kinetic-fokker-planck} belong to $L^{\infty}([0,T], \mathcal{H}_\kin^*)$.
        \item For any $\zeta >0$, there exists $C_{\Gamma, T, \zeta}>0$ depending only on $\Gamma, T$ and $\zeta$ such that
        \begin{equation}
\label{eq:N-estimate}
    \mathbb{E}\left[ \sup_{t\in [0,T] }\norm{\nu^N_t -\nu_t} \right] \leq C_{\Gamma, T, \zeta} \left(\mathbb{E}\left[\norm{\nu^N_0 - \nu_0}\right] + \frac{1}{N^{1/2-\zeta}} \right)\,.
\end{equation}
    \end{itemize}   
Moreover, if $\nu_0\in\mathcal{P}(\bR^{2d})$ and $\lim_{N\to \infty} \norm{\nu^N_0 - \nu_0} =0$ in probability, then $\nu^N$ converges to $\nu\in C([0,T], \mathcal{P}(\bR^{2d}))$ in probability according to the norm $\Vert \cdot\Vert$. 
\end{theorem}

\begin{remark}
 Regarding the regularity of the interaction kernel  $\Gamma$, the notion of ``sufficiently regular" introduced in Section \ref{s:kinetic} refers to kernels that possess slightly higher regularity than merely being bounded and Lipschitz continuous in both variables. This assumption is primarily motivated by the structure of the Hilbert space $\mathcal{H}_\kin^*$ , which is used to embed weak solutions of equation \eqref{kinetic-fokker-planck}, and by the requirement that this space be stable under pointwise products (see Proposition \ref{pro:product-besov}). From this viewpoint, our setting remains within the standard framework of propagation of chaos \cite{sznitman_topics_1991}, with the principal novelty being the relaxation of the usual independence assumption on the initial data. More general classes of kernels with weaker regularity than bounded–Lipschitz have been analyzed in \cite{hao_singular_2024} in the context of existence and uniqueness for equation \eqref{kinetic-fokker-planck}; however, establishing a direct link between their approach and ours appears to be a substantially more difficult problem.
\end{remark}

To prove Theorem \ref{thm:informal-lln}, we show that there exists a stochastic partial differential (SPDE) equation satisfied by the empirical measure (Proposition \ref{p:mild-formulation}) that is then compared with the limit PDE \eqref{kinetic-fokker-planck}. In order to do so, we make use of the kinetic semigroup (see Section \ref{s:kinetic}) and rewrite both equations in mild form (Definition \ref{d:weak-mild-solution}).
We study their difference in a suitable class of Hilbert spaces (see Section \ref{s:preliminaries}) which is built ad-hoc to exploit the peculiar kinetic setting.
The remainder term of this comparison is a stochastic convolution arising from the Brownian motions: while morally it has a vanishing quadratic variation, classical martingale tools cannot be applied because the stochastic convolution with the semigroup is not a martingale. More precisely, as the semigroup convolution appearing in the integrand involves the extremal time, this stochastic process is not adapted to the filtration. A careful use of Kolmogorov continuity theorem, and more precisely the Garsia-Rodemich-Rumsey lemma (Lemma \ref{lem:grr}), is needed (Lemma \ref{lem:pathwise_bound}).
Technically speaking, our main contribution is the characterization of both the kinetic semigroup and the underlying space in terms of Fourier (Proposition \ref{p:fourier-z}, Lemma \ref{lem:pathwise_bound}) which allows to prove the well-posedness of the SPDE  satisfied by the empirical measure (Proposition \ref{p:mild-formulation}), together with the use of recent results in the field of singular non-linear kinetic PDEs to obtain general regularization properties of the kinetic semigroup (Lemma \ref{lem:semigroup-time}).

The presented approach naturally extends to a broader class of interacting particle systems described by
\begin{equation}
    \label{eq:kinetic-diffusions2}
 \left\{ \; \,\begin{aligned}
    &\dd x^{i,N}_t = v^{i,N}_t \dd t,\\&
    \dd v^{i,N}_t = F(x^{i,N}_t,v^{i,N}_t) \dd t+\left(\frac{1}{N}\sum_{j\neq i} \Gamma((x^{i,N}_t,v^{i,N}_t),(x^{j,N}_t,v^{j,N}_t))\right)\dd t + \sigma \dd B^i_t
     \end{aligned}\right.\,
\end{equation}
under some regularity assumption on $F$. This simplification is deliberately chosen to highlight the core ideas and techniques without additional complications.

\subsection{Comparison with the existing literature}

Studying the action of a semigroup in the evolution of an interacting particle system has been used in several works both in the deterministic and stochastic setting, classical references are \cite{flandoli_uniform_2019, cardaliaguet_master_2019, delarue_master_2019}. The closest works that inspired our approach are given by \cite{bertini_synchronization_2014, bechtold_law_2021}. All cited works focus on the analytic semigroup issued by the Laplacian operator (or, more generally, by a symmetric and positively defined operator), while in our setting the semigroup arises from the kinetic setting and is not analytic.

The use of Hilbert spaces to study the empirical measure evolution was pioneered by Métivier \cite{metivier_weak_1987} and later employed in several works see, e.g., \cite{fernandez_hilbertian_1997, bertini_synchronization_2014} and references therein. In \cite{fernandez_hilbertian_1997}, the initial conditions are supposed to be IID and concentrated in space with finite $(4d + 1)$ moment to exploit Sobolev embeddings between weighted  spaces, while \cite{bertini_synchronization_2014} focuses on a specific model on the one-dimensional torus $\bT$, drastically simplifying the tools needed for showing the finite time horizon convergence.
The class of anisotropic spaces is well-known \cite{triebel_theory_2006}, yet its characterization as Besov space and application to the kinetic semigroup is recent and due to the works \cite{hao_schauder_2020, zhang_cauchy_2024, pascucci_sobolev_2024}, see also the references therein.
The classical choice of Sobolev spaces weighted by a measure which is invariant for the dynamics, see, e.g. \cite{villani_hypocoercivity_2006}, is not helpful in this context and rather complicates the estimates on the stochastic convolution. Moreover, this choice restricts the class of systems to those where the interaction $\Gamma$ is issued by a potential.

To the authors' knowledge, the only work addressing a similar question is given by \cite{bechtold_law_2021}, where the authors address a law of large numbers for non-kinetic case, using a mixture of analytic semigroup theory and rough path techniques. Our work can be seen as an extension of this work both in the result, where we are able to give a precise estimate for every $N$ and not just a LLN, and the techniques which go beyond the analytic case.
In addition, the proof presented here is probabilistic and does not require pathwise estimates, positively answering to the question raised in \cite[Remark 3.7]{bechtold_law_2021}.

\subsection{Organization of the work}

The next section presents the class of spaces in which the dynamics of the empirical measure is studied as well as the classical Garsia-Rodemich-Rumsey lemma. Section \ref{s:kinetic} presents the kinetic semigroup and its properties when acting on test functions, it ends with the definition of mild solutions to Equation \eqref{kinetic-fokker-planck} and a uniqueness result. Section \ref{s:lln} contains the main results on the well-posedness of the SPDE, the bound on the stochastic convolution and Theorem \ref{thm:lln} with its proof.

\medskip

\bigskip
\noindent \textbf{Acknowledgements:}
C.B. gratefully acknowledges funding support from the ERC Starting Grant Low Regularity Dynamics via Decorated Trees (LoRDeT). The views and opinions expressed are, however, those of the author(s) only and do not necessarily reflect those of the European Union or the European Research Council. Neither the European Union nor the granting authority can be held responsible for them. F.C. was supported by the NWO (Dutch Research Organization) grant OCENW.KLEIN.083. Both authors would like to thank Avi Mayorcas and Harprit Singh for their helpful references to the literature on hypoelliptic operators.
\bigskip

\section{Preliminaries}
\label{s:preliminaries}

\subsection{Kinetic H\"older and Sobolev spaces}
In this section, we recall the basic definitions and properties  of kinetic H\"older  and Sobolev spaces as a special cases of anisotropic Besov spaces \cite[Chapter 5]{triebel_theory_2006} in the context of the kinetic semigroup from \cite{hao_singular_2024}. 

For $(x,v), (y,w)\in\bR^{2d}$, we introduce the following kinetic distance over $\bR^{2d}$:
\begin{equation*}
|(x,v)-(y,w)|_{\mathtt{k}}:=|x-y|^{1/3}+|v-w|\, ,
\end{equation*}
where $ |\cdot|$ denotes the Euclidean norm in $\bR^d$.  Even though the expression of the left-hand side of the kinetic distance does not define a norm (because it is not homogeneous), we keep this notation for practical reason together with the shorthand expression $|(x,v)|_{\mathtt{k}} :=|(x,v)-(0,0)|_{\mathtt{k}}$.  Similarly for
$r>0$ and $(x,v)\in\bR^{2d}$, we introduce the corresponding closed balls:
\begin{equation*}
B^{\mathtt{k}}_r(x,v):=\{(y,w)\in\bR^{2d}\colon |(x,v)-(y,w)|_{\mathtt{k}}\leq r\}\,, \quad  B^\kin_r:=B^\kin_r((0,0))\,.
\end{equation*}
Under this new topology it is not difficult to see that $B^\kin_1$ has  Hausdorff dimension $4d$. This effective dimension will replace $2d$ in all our functional analytic embeddings. Once the topology is fixed we  introduce the corresponding partition of the unity,  by fixing a radial function $\phi^\kin_{0} \in C^{\infty}_c(\bR^{2d})$, with $0\leq \phi^\kin_{0}\leq 1$ and such that
\begin{equation*}
\phi^\kin_{-1}(\xi)=1\ \mathrm{for}\ \xi\in B^\kin_{1/2}\,,\quad  \phi^\kin_{-1}(\xi)=0\ \mathrm{for}\ \xi\notin B^\kin_{2/3}.
\end{equation*}
Then for  $\xi=(\xi_1,\xi_2)\in\bR^{2d}$ and $j\geq 0$ integer, we set
\begin{equation*}
\phi^\kin_j(\xi):=\phi^\kin_{-1}(2^{-\kin (j+1)}\xi)-\phi^\kin_{-1}(2^{-\kin j}\xi) \quad \text{ where }\quad 2^{-\kin j}\xi:=(2^{-3 j}\xi_1,2^{- j}\xi_2)\,.
\end{equation*}
 one has that $
\phi^\kin_j(\xi)=\phi^\kin_0(2^{-\kin j}\xi)\geq 0
$
and  we have the properties:
\begin{equation*}
\mathrm{supp}\,\phi^\kin_j\subset B^\kin_{2^{j+1}/3}\setminus B^\kin_{2^{j-1}}\,, \quad  \sum_{j=-1}^n \phi^\kin_j(\xi)=\phi^\kin_{-1}(2^{-\kin (n+1)}\xi) \overset{n\uparrow\infty}\longrightarrow 1\,.
\end{equation*}

In the following we will also need to use the Fourier transform over $\bR^{2d}$. We let $\mathcal{S}(\bR^{2d})$ be the space of all Schwartz functions on $\bR^{2d}$ and $\mathcal{S}'(\bR^{2d})$ the space of (tempered) distributions over $\bR^{2d}$.  For any $f \in \mathcal{S}(\bR^{2d})$  the Fourier transform $\cF(f) = \hat{f}$ and its inverse are defined by
\begin{align}
&\cF(f)(\xi, \eta)=\hat{f}(\xi, \eta)=\int_{\bR^{2d}} \exp\left(-i (\xi \cdot x+ \eta \cdot v) \right) f(x,v) \dd x \dd v \,,\\ &\cF^{-1}(f)(x, v)= \check f(x, v)=\frac{1}{(2\pi)^{2d}} \int_{\bR^{2d}} \exp\left(i (\xi \cdot x+ \eta \cdot v) \right)f(\xi,\eta) \dd \xi \dd \eta \,,
\end{align}
where $\cdot$ is the Euclidean  standard product over $\mathbb{R}^d$. 
This operator extends classically by duality over $\mathcal{S}'(\bR^{2d})$. For any integer $j\geq -1$, the block operator  $\mathcal{R}^\kin_j$  is defined by
\begin{align}\label{Ph0}
\mathcal{R}^\kin_jf(x):=(\phi^\kin_j\hat{f})\check{\ }(x)=\check{\phi}^\kin_j*f(x)\,,
\end{align}
 for any  $f\in \mathcal{S}'(\bR^{2d})$ where the convolution is understood in the distributional sense. In particular, for $j\geq 0$ one has
\begin{equation*}\mathcal{R}^\kin_jf(x,v)=2^{4d j}\int_{\bR^{2d}}\check{\phi}^\kin_{0}(2^{\kin j}(y,w))f(x-y,v-w)\dd y\dd w\,.\end{equation*}
 
From this block operators we introduce the  class of kinetic Besov spaces.
\begin{definition}\label{def:anisotropic-besov-space}
For $s\in\mathbb{R}$ and $p,q\in[1,\infty]$, the $\bR^{2d}$  kinetic Besov space is given by
\begin{equation*}
\mathcal{B}^{s}_{p, q;\kin}(\bR^{2d}):=\left\{f\in \mathcal{S}'(\bR^{2d}): \|f\|_{\mathcal{B}^{s}_{p, q;\kin}}
:=\left(\sum_{j\geq-1}\big(2^{sj}\|\mathcal{R}^\kin_{j}f\|_{L^p}\big)^q \right)^{1/q} < \infty\right\}.
\end{equation*}
We call the vector spaces $\mathcal{H}^s_\kin(\bR^{2d})= \mathcal{B}^{s}_{2, 2;\kin}(\bR^{2d})$ and $\mathcal{C}^{s}_\kin (\bR^{2d})=\mathcal{B}^{s}_{\infty, \infty;\kin}(\bR^{2d})$ the  kinetic Sobolev and H\"older space, respectively.
\end{definition}
As in the classical theory, $\mathcal{B}^{s}_{p, q;\kin}(\bR^{2d})$ is a Banach space. Moreover, we can easily extend this definition for any vector valued distribution $f\in \mathcal{S}'(\bR^{2d}; \mathbb{R}^m)$ by simply working component-wise. We list the main properties related  to inclusion properties and how the operation of derivatives modifies the exponents. For a proof of these facts, see e.g. \cite[Lemma 2.4, Theorem 2.6]{hao_singular_2024}. % 
\begin{proposition}\label{pro:nabla-f}
 The spaces  $\mathcal{B}^{s}_{p, q;\kin}(\bR^{2d})$ satisfy the following properties:
 \begin{itemize}
     \item Let $p,q\in [1, \infty]$, $s\in \mathbb{R}$, $k,m \geq 0$ integer. Then there exists a constant  $C>0$ such that for any $i,j\in \{1,\ldots, d\}$ the operator $\partial_{v_j}^m\partial^k_{x_i}$, the $k$-order derivative along  $x_i$ composed with the $m$-order gradient along  $v_j$  satisfies
\begin{equation}\label{eq:partial-anisotropic-inequality}
\|\partial_{v_i}^m\partial^k_{x_j}f\|_{\mathcal{B}^{s'}_{p,q;\kin}}\leq C \|f\|_{\mathcal{B}^{s}_{p,q;\kin}}\,,
\end{equation}
    with $s'=  s-3k- m$. That is $\partial_{v_i}^m\partial^k_{x_j}$ is a continuous operator from $\mathcal{B}^{s}_{p, q;\kin}(\bR^{2d})$ into $\mathcal{B}^{s'}_{p, q;\kin}(\bR^{2d})$.

     \item  Let $q\in [1, \infty]$ and $s\in\mathbb{R}$. for any $1\leq r\leq p\leq\infty$ there exists a constant  $C>0$  such that 
     \begin{equation}\label{Ber1}
\|f\|_{\mathcal{B}^{s'}_{p,q;\kin}}\leq C \|f\|_{\mathcal{B}^{s}_{r,q;\kin}}\,,
\end{equation}
 with $s'=  s-4d(\frac{1}{r}-\frac{1}{p})$. That is the embedding $ \mathcal{B}^{s}_{r,q;\kin}(\bR^{2d})\subset \mathcal{B}^{s'}_{p,q;\kin}(\bR^{2d})$ is continuous.
 \item   Let $p\in [1, \infty]$ and $s\in\mathbb{R}$. For any $1 \leq  q_1 \leq q_2\leq \infty$
\begin{equation}\label{Ber2}
\|f\|_{\mathcal{B}^{s}_{p,q_2;\kin}}\leq  \|f\|_{\mathcal{B}^{s}_{p,q_1;\kin}}\,.
\end{equation}
 That is the embedding $ \mathcal{B}^{s}_{p,q_1;\kin}(\bR^{2d})\subset \mathcal{B}^{s}_{p,q_2;\kin}(\bR^{2d})$ is continuous.
 \end{itemize}
\end{proposition}
The spaces $\mathcal{B}^{s}_{p,q_1;\kin}$ and in particular the H\"older ones $\mathcal{C}^{s}_\kin (\bR^{2d})$ when $s>0$ have also a characterisation in terms of finite increments and the standard theory of H\"older spaces. For  this reason, we introduce  the first order difference operator. For  $f\colon  \bR^{2d} \to \bR$ and $h\in \bR^{2d}$, we set
\begin{equation*}
\delta_hf(x):=f(x+h)-f(x),
\end{equation*}
and for $N\geq 1$ integer, the $N$-order difference operator is defined recursively by
\begin{equation*}
\delta^{(N)}_hf(x):= \delta_h\delta^{(N-1)}_hf(x).
\end{equation*}
By induction, it is easy to see that for any $h\in\bR^{2d}$
\begin{align}\label{Def8}
\delta^{(N)}_hf(x)=\sum^{N}_{k=0}(-1)^{N-k}\binom{N}{k} f(x+kh)\,.
\end{align}
A first characterisation comes from 
\cite[Theorem 2.7]{hao_singular_2024}, \cite[Lemma 2.8]{zhang_cauchy_2024}.
 
\begin{theorem}\label{thm:besov-increment-characterisation}
 For any $s\in(0,\infty)$ and $p,q\in[1,\infty]$, using the notation $[s]$ for the integer part of $s$, the  quantities
\begin{align}\label{eq_norm}
\|f\|_{\widetilde{\mathcal{B}}^{s}_{p,q;\kin}}&:=\left(\int_{B^\kin_1}
\left(\frac{\big\|\delta^{([s]+1)}_{h}f\big\|_{L^p}}{|h|_\kin^{s}}\right)^q\frac{\dd h}{|h|_\kin^{4d}}\right)^{1/q}
+\|f\|_{L^p}, &p,q\neq \infty\\ \label{eq_norm2}
\|f\|_{\widetilde{\mathcal{C}}^{s}_{\kin}}&:= \sup_{\substack{h\in \mathbb{R}^{d}\\ h\neq 0}}\frac{\big\|\delta^{([s]+1)}_{h}f\big\|_{L^\infty}}{|h|_{\kin}^{s}} +\|f\|_{L^\infty}, & p=q= \infty
\end{align}
are an equivalent norm for $\mathcal{B}^{s}_{p,q;\kin}(\bR^{2d})$. Moreover, in the case $ p=q=\infty$, using the partial increment operator
  \[
\delta_{h,1}f(x):=f(x+h, v)-f(x,v)\,, \quad \delta_{h,2}f(x):=f(x, v+h)-f(x,v)\,
\]
the quantity 
\begin{equation}\label{second_charac_holder}
\|f\|_{\mathcal{C}'^{s}_{\kin}}:= \|f\|_{L^\infty}+\sup_{\substack{h\in \mathbb{R}^{d}\\ h\neq 0}}\frac{\big\|\delta^{([s]+1)}_{h,1}f\big\|_{L^\infty}}{|h|^{s/3}} + \sup_{\substack{h\in \mathbb{R}^{d}\\ h\neq 0}}\frac{\big\|\delta^{([s]+1)}_{h,2}f\big\|_{L^\infty}}{|h|^{s}}\,,
\end{equation}
is an equivalent norm for $\mathcal{C}^{s}_\kin (\bR^{2d})$.
\end{theorem}
Furthermore, the space $\mathcal{C}^{s}_\kin (\bR^{2d})$  when $s\in(0,\infty)$ is non integer, admits an  even more explicit characterisation in terms of standard H\"older spaces whose derivatives are adjusted according to the kinetic scaling. For any couple of multi-indices $\alpha, \beta\in \mathbb{N}^d$, $\alpha=(\alpha_1\,, \ldots, \alpha_d)$, $\beta=(\beta_1\,, \ldots, \beta_d)$ we use the shorthand notation \[\partial_{x}^{\alpha }\partial^{\beta}_{v}f= \partial_{x_1}^{\alpha_1}\ldots\partial_{x_d}^{\alpha_d}\partial_{v_1}^{\beta_1}\ldots\partial_{v_d}^{\beta_d}f\]
for any $f\in \mathcal{S}(\bR^{2d})$. To each couple of multi-indices we associate the kinetic scaling
\[|(\alpha,\beta)|_\kin= 3 \sum_i^d\alpha_i + \sum_i^d\beta_i\]
This quantity modifies  the number of admissible derivatives that we can consider for the kinetic H\"older spaces.
\begin{proposition}\label{prop_holder}
    Let $s>0$ be non integer. Then a function $f\in \mathcal{C}^{s}_{\kin}(\bR^{2d})$ if and only if for all $|(\alpha,\beta)|_\kin\leq [s]$ the function $\partial_{x}^{\alpha }\partial^{\beta}_{v}f$ is a continuous bounded function and the quantity 
\[\|f\|_{C^{s}_{\kin}}= \max_{\substack{\alpha,\beta\, \\|(\alpha,\beta)|_\kin\leq [s]}} \Vert\partial_{x}^{\alpha }\partial^{\beta}_{v}f\Vert_{L^\infty}+ \max_{\substack{\alpha,\beta\,\\ |(\alpha,\beta)|_\kin= [s]}} \sup_{\substack{(x,v),(y,w)\in \mathbb{R}^{2d},\\ (x,v)\neq (y,w)}} \frac{|\partial_{x}^{\alpha }\partial^{\beta}_{v}f(x,v)- \partial_{x}^{\alpha }\partial^{\beta}_{v}f(y,w)|}{|(x,v)-(y,w)|_{\mathtt{k}}^{s-[s]}}\]
is an equivalent norm for $\mathcal{C}^{s}_{\kin}(\bR^{2d})$. Moreover,  there exists a constant $C>0$  such that for any  integer $k\leq [s] $ one has 
\begin{equation}\label{eq_growth_delta}
\|\delta_h^{(k)}f\|_{L^\infty}\leq  C  \|f\|_{C^{s}_{\kin}} |h|_\kin^{k}\,,\quad \|\delta_h^{([s]+1)}f\|_{L^\infty}\leq  C  \|f\|_{C^{s}_{\kin}} |h|_\kin^{s}.
\end{equation}
\end{proposition}
\begin{proof}
 The first part of the statement is indeed a classical result in the literature of Besov spaces, when $\mathbb{R}^{2d}$ has the standard Euclidean metric and $\mathcal{C}^{s}_{\kin}(\bR^{2d})$ is replaced by $\mathcal{C}^{s}(\bR^{2d})$ the standard H\"older spaces. We briefly  sketch the characterisation when $s\in (0,1)$, since the general case follows by simply applying enough derivatives to a function $f$ according to Proposition \ref{pro:nabla-f}. Given a function $f\colon \mathbb{R}^{2d}\to \mathbb{R}$ satisfying $\|f\|_{\mathcal{C}^{s}_{\kin}}<\infty$   to show that there exists a constant $C>0$ such that for any $j\geq -1$
 \[\|(\mathcal{R}^\kin_j f)\|_{L^{\infty}}\leq C \|f\|_{\mathcal{C}^{s}_{\kin}}2^{sj},\]
we just use the convolution representation of $\mathcal{R}^\kin_j$ and the fact that $\int \check{\phi}^\kin_0=0$ and the result follows by elementary estimates. Concerning the inverse inclusion, for any given $f\in \mathcal{C}^{s}_{\kin}(\mathbb{R}^{2d})$ by writing it as
\[f=\sum_{j\geq -1}\mathcal{R}^\kin_j f\]
for any $(x,v),(y,w)\in \mathbb{R}^{2d}$ with $ (x,v)\neq (y,w)$ there exists a constant $C'>0$ such that 
\begin{align*}
    &|f(x,v)- f(y,w)|\leq \\&\leq \sum_{\substack{j\geq -1\\|(x,v)-(y,w)|_{\mathtt{k}}\leq 2^{-j}}}|\mathcal{R}^\kin_j f(x,v)- \mathcal{R}^\kin_j f(y,w)|+\sum_{\substack{j\geq -1\\ |(x,v)-(y,w)|_{\mathtt{k}}> 2^{-j}}} |\mathcal{R}^\kin_j f(x,v)- \mathcal{R}^\kin_j f(y,w)|\\&\leq \sum_{\substack{j\geq -1, i,l=1, \ldots d\\|(x,v)-(y,w)|_{\mathtt{k}}\leq 2^{-j}}}|x-y|\|\partial_{x_i}\mathcal{R}^\kin_j f\|_{L^{\infty}}+ |v-w|\|\partial_{v_l}\mathcal{R}^\kin_jf\|_{L^{\infty}}+C'|(x,v)-(y,w)|_{\mathtt{k}}^s\|f\|_{\mathcal{C}^{s}_{\kin}}\,.
\end{align*}
Then it is sufficient to apply the properties of the derivatives from Proposition \ref{pro:nabla-f} to obtain the complete estimate. Concerning the estimate \eqref{eq_growth_delta} the second inequality follows immediately from Theorem \ref{thm:besov-increment-characterisation}
therefore one needs to control $\|\delta_h^{(k)}f\|_{L^\infty}$ for $1\leq k\leq [s]$ integer. Using now the partial increment operators, for any $h\in \mathbb{R}^{2d}$, $h=(h_1,h_2)$ we can write
\[\delta_h f(x,v)=(\delta_{h_1,1}\circ T_{h_2,2}+ \delta_{h_2,2})f(x,v)\]
where $T_{h_2,2}f(x)=f(x,v+h_2)$ is a translation in the velocity variable. Since the operators  $\delta_{h_1,1}$ and $T_{h_2,2}$ commute with $\delta_{h_2,2}$  and $T_{h_2,2}$ commute with $\delta_{h_1,1}$ we have
\[\delta_h^{(k)}f(x,v)=\sum_{l=0}^k\binom{k}{l}\delta_{h_1,1}^{(l)}\circ T_{l h_2,2} \circ\delta_{h_2,2}^{(k-l)} f(x,v)\,,\]
where we use the convention $\delta_h^{(0)}=\text{Id}$. Applying now $k-l$ times the intermediate value theorem on the variables $v$ and $(k-l)/3$ times on the variables $x$ together with the H\"older property in the $x$ variables there exists a constant $M>0$ such that
\[|\delta_{h_1,1}^{(l)}\circ T_{l h_2,2} \circ\delta_{h_2,2}^{(k-l)} f(x,v)|\leq M \|f\|_{C^{s}_{\kin}} |h_1|^{l/3}|h_2|^{k-l}\,,\]
from which we derive \eqref{eq_growth_delta} thanks to the binomial theorem.
\end{proof}

Passing to the kinetic Sobolev spaces $\mathcal{H}^s_\kin(\bR^{2d})$, like in the classical case, they have an intrinsic Hilbert structure that can be described in terms of the Fourier transform.

\begin{proposition}\label{pro:anisotropic-hilbert}
Let $s \in \mathbb{R}$. Then the kinetic Sobolev spaces $\mathcal{H}^s_\kin(\bR^{2d})$ can be equivalently defined by 
\begin{equation}
\label{def:anisotropic-hilbert}
    \mathcal{H}^s_\kin (\bR^{2d}) = \left\{ f\in \mathcal{S}'(\bR^{2d})\colon  \int_{\bR^{2d}} (1 + |\xi|^{2/3} + |\eta|^2)^s |\hat{f}(\xi, \eta)|^2 \dd \xi \dd \eta < \infty \right\},
\end{equation}
with scalar product defined by
\begin{equation*}
    ( f, g )_{s}  = \int (1 + |\xi|^{2/3} + |\eta|^2)^s \hat f(\xi, \eta) \overline{\hat g(\xi, \eta)} \dd \xi\dd \eta\,.
\end{equation*}
and norm $\Vert f\Vert_s^2=  ( f, f )_{s}$. Moreover for any $s_1<s_2$ we have the estimate
\begin{equation}
\label{eq:continuity_embedding}
\Vert f\Vert_{s_1}\leq  \Vert f\Vert_{s_2}
\end{equation}
which corresponds to the continuity embedding $\mathcal{H}^{s_2}_\kin(\bR^{2d})\subset \mathcal{H}^{s_1}_\kin(\bR^{2d}) $.
\end{proposition}

\begin{proof}
Introducing the function $J\colon\bR^{2d}\to \mathbb{R}$, $J(\xi, \eta)= (1+ |\xi|^2) 
^{\frac{1}{6}}+ (1+ |\eta|^2) ^{\frac{1}{2}}\,, $
 and the invertible operator $ J_s\colon \mathcal{S}'(\bR^{2d})\to \mathcal{S}'(\bR^{2d})$ defined by 
 \begin{equation*}\widehat{J_s(f)}= J(\xi, \eta)^s \hat{f}(\xi, \eta)\,,\end{equation*}
 it follows from \cite[Lemma 2.4, Remark 2.5]{hao_singular_2024} that $J_s$ is an isometry and one has
 \begin{equation*}\mathcal{H}^s_\kin(\bR^{2d})= J_s (\mathcal{B}^{0}_{2, 2;\kin}(\bR^{2d})) = J_s (L^2(\bR^{2d}))\,,\end{equation*} where the last equality follows directly from  Definition \ref{def:anisotropic-besov-space} and the standard properties of the Fourier transform. Then the result follows by simply observing that there exists two constant $C_{s,d} , c_{s,d} > 0$  depending only on $d$ and $s$ such that
 \begin{equation*} c_{s,d}(1 + |\xi|^{2/3} + |\eta|^2)^s\leq |J(\xi, \eta) |^{2s}\leq C_{s,d}(1 + |\xi|^{2/3} + |\eta|^2)^s\,.
 \end{equation*}
The estimate \eqref{eq:continuity_embedding} and the associated embedding  are true for any $f\in \mathcal{S}(\mathbb{R}^{2d})$ and the result holds using a classical density argument. 
\end{proof}
From this characterisation, we can  describe and identify the dual space $(\mathcal{H}^s_\kin(\bR^{2d}))^{*}$  with $\langle f,g\rangle_{(\mathcal{H}^s_\kin)^{*},\mathcal{H}^s_\kin }$  the standard duality pairing and its dual norm
\begin{equation*}\Vert f\Vert_{(\mathcal{H}^s_\kin)^{*}}:=\sup_{g\colon \Vert g\Vert_s\leq 1}\langle f,g\rangle_{(\mathcal{H}^s_\kin)^{*},\mathcal{H}^s_\kin }\,,\end{equation*}
in terms of  $\mathcal{H}^{-s}_\kin(\bR^{2d})$. For this reason, we use the shorthand notations 
\begin{equation*}\Vert f\Vert_{- s}= \Vert f\Vert_{(\mathcal{H}^s_\kin)^{*}}\,, \quad  \langle f,g\rangle_{-s,s}=\langle f,g\rangle_{(\mathcal{H}^s_\kin)^{*},\mathcal{H}^s_\kin }\,.\end{equation*}
Using the natural isometry of $(\mathcal{H}^s_\kin)^{*}$ with $\mathcal{H}^{-s}_\kin$ we will always have 
\begin{equation}\label{CS_sobolev}
|\langle f,g\rangle_{-s,s}|\leq \Vert f\Vert_{-s}\Vert g\Vert_{s}
\end{equation}
Thanks to this identification, we can easily describe  $\cP(\bR^{2d})$, the space of probability measure over $\bR^{2d}$ as a bounded  subset of $(\mathcal{H}^{s}_\kin(\bR^{2d}))^*$.

\begin{lemma}
    \label{lem:embedding-proba}
For all $s>2d$, one has the following continuous embedding
    \begin{equation}
    \label{a:embBes}
    \mathcal{H}^s_\kin (\bR^{2d})\subset C_b(\bR^{2d})\,,
    \end{equation}
 with  $ C_b(\bR^{2d})$ the set of continuous and bounded functions $f\colon \bR^{2d}\to \bR$. Moreover there exists a constant $M>0$ depending only on $s$ such that
    \begin{equation}
    \label{a:embPro}
    \sup_{\mu \in \cP(\bR^d)} \norm{\mu}_{-s} \leq M\,.
    \end{equation}
\end{lemma}
\begin{proof}
The proof is a standard adaptation of 
\cite[Lemma A.2]{bechtold_law_2021} in the context of kinetic setting. We repeat it here for sake of completeness. Using the embeddings \eqref{Ber1}, \eqref{Ber2}  and Theorem \ref{thm:besov-increment-characterisation} one has 
\begin{equation}
\label{def:besov-sobolev-embedding}
    \mathcal{H}^s_{\kin} (\bR^{2d})=  \mathcal{B}^{s}_{2,2;\kin} (\bR^{2d})\subset \mathcal{B}^{s-2d}_{\infty,2;\kin} (\bR^{2d})\subset  \mathcal{C}^{s-2d}_{\kin}(\bR^{2d})\subset C_b(\bR^{2d})\,,
\end{equation}
where the last embedding is true because $s-2d>0$. In order to prove \eqref{a:embPro}, we can look at any element $\mu \in \cP(\bR^{2d})$ as an element of $(C_b(\bR^{2d}))^*$ and therefore of $(\mathcal{H}^s_\kin(\bR^{2d}))^{*}$. Moreover for any  $f\in \mathcal{H}^s_{\kin} (\bR^{2d})$ one has
\begin{equation*}\langle \mu, f\rangle_{-s,s}=\int_{\bR^{2d}}f(x,v)\mu(\dd x, \dd v)= \frac{1}{(2\pi)^{2d}}\int_{\bR^{2d}}\hat{f}(\xi,\eta)\hat{\mu}(\xi, \eta)\dd \xi \dd \eta \,,\end{equation*}
where 
\begin{equation*}\hat{\mu}(\xi, \eta)= \int_{\bR^{2d}} \exp\left(i (\xi \cdot x+ \eta \cdot v) \right)  \mu(\dd x, \dd v)\end{equation*}
as a simple consequence of Fourier inversion theorem, Fubini theorem and the condition $s>0$. Multiplying and diving by the weight $(1 + |\xi|^{2/3} + |\eta|^2)^{s/2} $ and applying Cauchy-Schwarz inequality, for any $f\in \mathcal{H}^s_{\kin} (\bR^{2d})$ with  $\Vert f\Vert_s\leq 1$ one has
\begin{equation*}|\langle \mu, f\rangle_{-s,s}|\leq  \frac{1}{(2\pi)^{2d}}\left(\int_{\bR^{2d}} (1 + |\xi|^{2/3} + |\eta|^2)^{-s}   |\hat{\mu}(\xi, \eta)|^2\dd \xi \dd \eta \right)^{1/2}\,.\end{equation*}
Since $|\hat{\mu}(\xi, \eta)|^2\leq 1$ uniformly on $\xi, \eta$, the result follows by simply checking that the integral 
\begin{equation*}
\int_{\bR^{2d}} (1 + |\xi|^{2/3} + |\eta|^2)^{-s}   \dd \xi \dd \eta
\end{equation*}
is finite as $s>2d$.
\end{proof}
We can finally combine the properties of kinetic Sobolev and H\"older spaces to derive a continuity result of the product.
\begin{proposition}
    \label{pro:product-besov}
  Let $f\in \mathcal{H}^s_{\kin} (\bR^{2d})$,$g\in \mathcal{C}^{r}_{\kin}(\bR^{2d})$ for some $s,r\geq 0$. Under the hypothesis $r>s$ and $r$ non-integer, the pointwise product  $gf$  belongs to $\mathcal{H}^s_{\kin} (\bR^{2d})$ and there exists a constant $C>0$ such that
 \begin{equation}\label{eq:prod_property} \Vert fg\Vert_s\leq C \Vert g\Vert_{\mathcal{C}^{r}_\kin}\Vert f\Vert_s.
 \end{equation}
\end{proposition}
\begin{proof}
Thanks to the embedding properties in Proposition \ref{prop_holder}  it is not restrictive to suppose $s <r< [s]+1$. Then the result will simply follow from by showing \eqref{eq:prod_property}. Using the equivalence property in Theorem \ref{thm:besov-increment-characterisation} and Proposition \ref{prop_holder} we will show that
 \begin{equation*}
\Vert fg\Vert_{\widetilde{\mathcal{H}}^{s}_{\kin}}^2\leq C \Vert g\Vert_{C^{r}_\kin}^2\Vert f\Vert_{\widetilde{\mathcal{H}}^{s}_{\kin}}^2
\end{equation*}
where $\Vert f\Vert_{\widetilde{\mathcal{H}}^{s}_{\kin}}^2$ is given by
\[
\Vert f\Vert_{\widetilde{\mathcal{H}}^{s}_{\kin}}^2=\int_{B^\kin_1}
\frac{\big\|\delta^{([s]+1)}_{h}(fg)\big\|_{L^2}^2}{|h|_\kin^{2s}}\frac{\dd h}{|h|_\kin^{4d}}
+\|fg\|_{L^2}^2\,.
\]
Using the simple Leibniz formula for finite increments the estimate \eqref{eq_growth_delta} there exists a constant $M>0$ such that 
\begin{equation*}
\begin{split}
    &\delta^{([s]+1)}_{h}(fg)(x)= \sum_{l=0}^{[s]+1}\binom{[s]+1}{l}\delta^{(l)}_{h}(g)(x)\delta^{([s]+1-l)}_{h}(f)(x+lh)\\
    &\leq M\sum_{l=0}^{[s]}\binom{[s]+1}{l}\Vert g\Vert_{C^{r}_\kin} |h|_\kin^{l}\delta^{([s]+1-l)}_{h}(f)(x+lh)+ \Vert g\Vert_{C^{r}_\kin} |h|_\kin^{r}f(x+([s]+1)h)
\end{split}
\end{equation*}
From this estimate, using the invariance of the $L^2$ norm by translation we derive that there exists another constant $M'>0$ such that
\[\big\|\delta^{([s]+1)}_{h}(fg)\big\|_{L^2}^2\leq M'\Vert g\Vert_{C^{r}_\kin}^2\left(\sum_{l=0}^{[s]}|h|_\kin^{2l}\big\|\delta^{([s]+1-l)}_{h}(f)\big\|_{L^2}^2 + |h|_\kin^{2r}\big\|f\big\|_{L^2}^2\right)\,.\]
from which we deduce that 
\begin{align*}
   & \int_{B^\kin_1}
\frac{\big\|\delta^{([s]+1)}_{h}(fg)\big\|_{L^2}^2}{|h|_\kin^{2s+4d}}\dd h\leq  \\&\leq  M'\Vert g\Vert_{C^{r}_\kin}^2\left(\sum_{l=0}^{[s]}\int_{B^\kin_1}
|h|_\kin^{2l-2s-4d}\big\|\delta^{([s]+1-l)}_{h}(f)\big\|_{L^2}^2\dd h + \int_{B^\kin_1}|h|_\kin^{2r-2s-4d}\big\|f\big\|_{L^2}^2\dd h\right)\,\\&\leq  M'\Vert g\Vert_{C^{r}_\kin}^2\left(\sum_{l=0}^{[s]}\Vert f\Vert_{\widetilde{\mathcal{H}}^{s-l}_{\kin}}^2 + \int_{B^\kin_1}|h|_\kin^{2r-2s-4d}\dd h\Vert f\Vert_{\widetilde{\mathcal{H}}^{s}_{\kin}}^2\right).
\end{align*}  
Then using the estimate \eqref{eq:continuity_embedding} in terms of the equivalent norms $\Vert \cdot\Vert_{\widetilde{\mathcal{H}}^{s}_{\kin}}$ and the simple fact that 
\[
\int_{B^\kin_1}|h|_\kin^{\alpha-4d}\dd h<\infty
\]
for any $\alpha>0$, we conclude.
\end{proof}

In addition to the properties of kinetic Sobolev and H\"older spaces, we recall a generalization of Kolmogorov's continuity criterion, the Garsia-Rodemich-Rumsey lemma. This is a classical result to study continuous path $\phi \colon [0, T]\to \mathcal{H}^s_{\kin}(\mathbb{R}^{2d})$ which will be used in Section \ref{s:lln} to handle the stochastic convolution arising from the Brownian motions in Equation \eqref{eq:kinetic-diffusions}. See, e.g., \cite[Theorem 2.1.3]{stroock_multidimensional_2006} for a proof.

\begin{lemma}[{Garsia-Rodemich-Rumsey Lemma}]\label{lem:grr}
Let $p$ and $\Psi$ be continuous, strictly increasing functions on $[0, \infty)$ such that $p(0) = \Psi(0) = 0$ and $\lim_{t \to \infty} \Psi(t) = \infty$. Given $T > 0$ and $\phi \in C([0, T], \mathcal{H}^s_{\kin}(\mathbb{R}^{2d}))$ for some $s\in \mathbb{R}$, if
\begin{equation*}
\int_0^T \int_0^T \Psi \left( \frac{\Vert\phi(t) - \phi(s)\Vert_{s}}{p(|t - s|)} \right) \dd s \, \dd t \leq B,
\end{equation*}
then for $0 \leq s < t \leq T$
\begin{equation*}
\Vert\phi(t) - \phi(s)\Vert_{s} \leq 8 \int_0^{(t - s)} \Psi^{-1} \left( \frac{4B}{u^2} \right) p(\dd u)\,.
\end{equation*}
\end{lemma}

\section{Kinetic mild PDE}\label{s:kinetic}

\subsection{The kinetic semigroup}

From now on and without loss of generality, we assume $\sigma=\sqrt{2}$.
Let $B_t$ be a $d$-dimensional standard Brownian motion. Define
\begin{equation*}
    (X_t, V_t):=\left(\sqrt{2}\int^t_0 B_s \dd s, \sqrt{2}B_t\right)
\end{equation*}
and the kinetic semigroup
\begin{equation}\label{def:kinetic-semigroup}
P_t f(x,v):=\mathbb{E} \left[ f(x+tv+X_t, v+V_t) \right]=(T_tp_t)*(T_tf)(x,v),
\end{equation}
where $T_t f(x,v):=f(x+tv,v)$ is the free transport in the velocity variable and $p_t$ is the density of $(X_t, V_t)$ given by
\begin{equation*}
p_t(x,v)=\left(\frac{2\pi t^4}{3}\right)^{-\frac{d}{2}}\exp\left(-\frac{3|x|^2+|3x-2tv|^2}{4t^3}\right).
\end{equation*}
It is easy to see that for any $f\in\mathcal{S}(\bR^{2d})$, one has  by It\^o's formula
\begin{align}\label{eq:kinetic-semigroup-ito}
 \partial_tP_t f=[\Delta_v+v\cdot\nabla_x]P_tf,
\end{align}
which also holds for $f\in\mathcal{S}'(\bR^{2d})$ in the distributional sense by duality. The choice of working with kinetic Sobolev spaces implies  the following simple yet important estimate about the kinetic semigroup $P_t$.

\begin{lemma}\label{lem:semigroup-time}
For any $T, s>0$, there exists a constant $C>0$ such that for all $t\in(0,T]$,
\begin{align}
\norm{\nabla_v P_t f}_{s}\leq \frac C{\sqrt{t}} \norm{f}_{s}.
\end{align}
\end{lemma}

\begin{proof}
The result is a consequence of the properties of the kinetic Sobolev spaces, notably equation \eqref{eq:partial-anisotropic-inequality} implies that
\begin{equation*}
\norm{\nabla_v P_t f}_{s}\leq C'\norm{P_t f}_{s+1}.
\end{equation*}
The proof is concluded if we prove that for any $p\in [1, \infty]$, it holds that
\begin{equation}\label{thesis_equation}
\|P_t f\|_{\mathcal{B}^{s+1}_{p,2;\kin}}\leq \frac C{\sqrt{t}}  \|f\|_{\mathcal{B}^{s}_{p,2;\kin}}.
\end{equation}
The previous equation is proven in \cite[Lemma 3.3]{zhang_cauchy_2024} in the case $q=\infty$, we adapt it to the case $p=2$. Let $\mathbb{N}^* = \{1, 2, \dots\}$ and define
\begin{equation*}
    \Theta^t_0:= \left\{l\in \mathbb{N}^* \colon 2^l \leq 2^4(1+t) \right\}
\end{equation*}
and for $j\geq 1$,
\begin{equation*}
    \Theta^t_j:= \left\{l\in \mathbb{N}^* \colon 2^l \leq 2^4(2^j +t2^{3j}), \; 2^j\leq 2^4(2^l+t2^{3l}) \right\}.
\end{equation*}
Thanks to the orthogonality of $\mathcal{R}^\kin_j$ outside $\Theta^t_j$ (see \cite[Lemma 6.7]{hao_schauder_2020}), we have that
\begin{equation*}
    \mathcal{R}^\kin_j P_t f = \mathcal{R}^\kin_j \Gamma_t p_t *\Gamma_t f = \sum_{l \in \Theta^t_j}  (\mathcal{R}^\kin_l \Gamma_t p_t) *(\Gamma_t \mathcal{R}^\kin_lf), \quad j\in\mathbb{N}^*.
\end{equation*}
Using Young's inequality and the fact that $\Gamma_t$ is just a translation, we have that
\begin{equation*}
    \begin{split}
        \|\mathcal{R}^\kin_j P_t f\|_{L^p} &\leq \|\mathcal{R}^\kin_j \Gamma_t p_t\|_{L^1} \sum_{l\in \Theta^t_j} \|\mathcal{R}^\kin_j f\|_{L^p}.
    \end{split}
\end{equation*}
Thanks to \cite[Lemma 3.1]{zhang_cauchy_2024}, we know that there exists a constant $C>0$, only depending on the dimension $d$, such that $\|\mathcal{R}^\kin_j \Gamma_t p_t\|_{L^1}\leq C t^{-1/2} \,2^{-j}$ for all $j \in \mathbb{N}$ and $t>0$. Thus
\begin{equation*}
    \begin{split}
        \|\mathcal{R}^\kin_j P_t f\|_{L^p} \leq C t^{-1/2} \,2^{-j} \sum_{l\in \Theta^t_j} \| \mathcal{R}^\kin_l f\|_{L^p}.
    \end{split}
\end{equation*}
Multiplying both sides by $2^{(s+1)j}$, taking the square and summing over $j\ge0$, one obtains that
\begin{equation*}
\sum_{j\ge0}
2^{2(s+1)j}\,\|\mathcal{R}^\kin_j P_t f\|_{L^p}^2
\;\le\;
C \,t^{-1}\,
\sum_{j\ge0}2^{2sj}\, \sum_{l,k \in \Theta^t_j} \|\mathcal{R}^\kin_l f\|_{L^p} \|\mathcal{R}^\kin_k f\|_{L^p}.
\end{equation*}

We now observe that there is an absolute constant $M>0$ (depending only on $d$ and $T$) such that uniformly in $l,k\in\mathbb{N}^*$  and $t\in (0, T]$
\begin{equation*}
\#\bigl\{j\ge0:\,l\in\Theta^t_j\text{ and }k\in\Theta^t_j\}\le M\,.
\end{equation*}
Hence,
\begin{equation*}
\begin{split}
&\sum_{j\ge0}2^{2sj}\sum_{l,k\in\Theta^t_j} \|\mathcal R^\kin_l f\|_{L^p}\,\|\mathcal R^\kin_k f\|_{L^p}
     \;\le\;
\sum_{l,k\ge0}\Bigl(\sum_{j\ge0}1_{\{l,k\in\Theta^t_j\}}\Bigr)
\,2^{sj}\|\mathcal R^\kin_l f\|_{L^p}\,2^{sk}\|\mathcal R^\kin_k f\|_{L^p}\\
& \le
M\sum_{l,k\ge0}2^{sl}\|\mathcal R^\kin_l f\|_{L^p}\,2^{sk}\|\mathcal R^\kin_k f\|_{L^p}
\;=\;
M\Bigl(\sum_{l\ge0}2^{sl}\|\mathcal R^\kin_l f\|_{L^p}\Bigr)^2
\;\le\;
M\sum_{l\ge0}2^{2sl}\|\mathcal R^\kin_l f\|_{L^p}^2,
\end{split}
\end{equation*}
where in the last step we used Cauchy–Schwarz on the $\ell^1$–sum in $l$.  Hence
\begin{equation*}
\sum_{j\ge0}
2^{2(s+1)j}\,\|\mathcal{R}^\kin_j P_t f\|_{L^p}^2
\;\le\;
C\,t^{-1}\;
\Bigl(\sum_{l\ge0}2^{2sl}\|\mathcal R^\kin_l f\|_{L^p}^2\Bigr)
\;=\;
\frac{C}{t}\,\|f\|^2_{\mathcal B^s_{p,2;\kin}}.
\end{equation*}
Taking square‐roots gives exactly \eqref{thesis_equation}.
\end{proof}

The kinetic semigroup has an explicit Fourier characterization.
\begin{lemma}
\label{lem:semigroup-fourier}
Let $f\in \mathcal{S}'(\bR^{2d})$, we have that

\begin{equation*}
    {P_tf}(x, v) = \frac 1{(2\pi)^{2d}}\int\int \exp\left(i (\xi \cdot x + \eta \cdot v)\right) G(t, \xi, \eta) \widehat{f}(\xi, \eta) \dd \xi \dd \eta,
\end{equation*}
with $G$ given by
\begin{equation}
\label{def:kinetic-kernel}
    G(t, \xi, \eta) = \exp \left( - \frac{t^3}{3} | \xi |^2 - \langle
   \xi, \eta \rangle t^2 - | \eta |^2 t \right).
\end{equation}
\end{lemma}
\begin{proof}
Observe that $\mathcal{F} (P_t f) (\xi, \eta) = \hat{u} (t, \xi, \eta)$ where
$\hat{u}$ solves the PDE
\begin{equation*} \partial_t \hat{u} = - \xi \nabla_{\eta} \hat{u} - | \eta |^2 \hat{u}
   \qquad \hat{u} (0, \xi, \eta) = \hat{f} (\xi, \eta) \end{equation*}
for any $\xi$.  This is a first order PDE whose explicit solution is given by
\begin{equation*}
\hat{u} (t, \xi, \eta) = \exp \left( - \frac{t^3}{3} | \xi |^2 - t^2\xi\cdot \eta - | \eta |^2 t \right)  \hat{f} (\xi, \eta)\,.
\end{equation*}
\end{proof}
We now turn to the properties of $G$.
\begin{lemma}
\label{lem:kernel-time-regularity}
The kernel $G$ defined in \eqref{def:kinetic-kernel} satisfies for $\xi, \eta \in \mathbb{R^d}$
\begin{equation*}
    |G(t-r, \xi, \eta) -G(u-r, \xi, \eta)|\leq \frac 14 |\eta|^2 (t-u),
\end{equation*}
for any choice of $0\leq u \leq r \leq t$.
\end{lemma}

\begin{proof}
We begin by analyzing the ratio which simplifies to
\begin{align*}
&\frac{G(t-r, \xi, \eta)}{G(u-r, \xi, \eta)} =\\&=
\exp\left( - \frac{(t-r)^3 - (u-r)^3}{3} |\xi|^2 - \left((t-r)^2 - (u-r)^2\right) \langle \xi, \eta \rangle - (t - u) |\eta|^2 \right).
\end{align*}
Now, we rewrite the exponent in a more structured form. Define
\begin{equation*}
\Delta_{t,u}(r) := (t-r)^3 - (u-r)^3, \quad
\delta_{t,u}(r) := (t-r)^2 - (u-r)^2.
\end{equation*}
Then the exponent becomes
\begin{equation*}
\begin{split}
    &- \frac{\Delta_{t,u}(r)}{3} |\xi|^2 - \delta_{t,u}(r) \langle \xi, \eta \rangle - (t - u) |\eta|^2 \\
&= - \frac{\Delta_{t,u}(r)}{3} \left| \xi + \frac{3\delta_{t,u}(r)}{2\Delta_{t,u}(r)} \eta \right|^2
+ (t - u)|\eta|^2 \left( -1 + \frac{3\delta_{t,u}(r)^2}{4\Delta_{t,u}(r)} \right).
\end{split}
\end{equation*}

It can be shown that the function
\begin{equation*}
f(r) := - \frac{\Delta_{t,u}(r)}{3} \left| \xi + \frac{3\delta_{t,u}(r)}{2\Delta_{t,u}(r)} \eta \right|^2
+ (t - u)|\eta|^2 \left( -1 + \frac{3\delta_{t,u}(r)^2}{4\Delta_{t,u}(r)} \right)
\end{equation*}
is decreasing in $r$, and hence for $r \leq u$, the exponent is bounded above by $- \frac{(t-u)}{4} |\eta|^2$.

Therefore,
\begin{equation*}
\frac{G(t-r, \xi, \eta)}{G(u-r, \xi, \eta)} \leq \exp\left( - \frac{(t - u)}{4} |\eta|^2 \right),
\end{equation*}
and thus,
\begin{equation*}
\begin{split}
    |G(t-r, \xi, \eta) - G(u-r, \xi, \eta)|
    &= G(u-r, \xi, \eta) \left| \frac{G(t-r, \xi, \eta)}{G(u-r, \xi, \eta)} - 1 \right| \\
    &\leq G(u-r, \xi, \eta) \left( 1 - \exp\left( - \frac{(t - u)}{4} |\eta|^2 \right) \right).
\end{split}
\end{equation*}
Using the inequality $1 - e^{-x} \leq x$ for $x \geq 0$, we conclude.
\end{proof}

\subsection{Weak-mild solution}
We use the properties of the kinetic semigroup to reformulate the kinetic Fokker-Planck equation \eqref{kinetic-fokker-planck} with $\sigma= \sqrt{2}$. Indeed, for any compactly supported smooth function $f\colon \mathbb{R}^{2d}\to \mathbb{R}$ and any probability solution $\nu\in C([0,T], \mathcal{P}(\bR^{2d}))$ we can multiply both sides of the equation  and formally provide integration by parts to obtain for any $t \in [0,T]$, the weak formulation  
\begin{equation}\label{weak_solution}
\langle\nu_t,f\rangle=\langle\nu_0,f\rangle+ \int_0^t \langle  \nu_s ,(v \cdot \nabla_x  +   \Delta_v)f+  (\Gamma*\nu_s)\cdot \nabla_vf\rangle \dd s\,,
\end{equation}
with $\langle\,\,,\rangle$ the usual pairing between functions and probability measures. In our setting we prefer to use the properties of the semigroup $P$ and kinetic Sobolev spaces $\mathcal{H}^s_{\kin} (\bR^{2d})$ with the dual pairing $\langle\, \,,\rangle_{-s,s}$ to disentangle the linear part of this equation from the non-linear part, obtaining the definition of weak-mild solutions, as introduced in \cite[Definition 2.1]{bechtold_law_2021}.

\begin{definition}[weak-mild solution]
\label{d:weak-mild-solution}
Let $\nu_0$ be an element in $(\mathcal{H}^{s}_{\kin} (\bR^{2d}))^{*}$ for some $s>0$. We call $\nu \in L^\infty([0,T], (\mathcal{H}^{s}_{\kin} (\bR^{2d}))^{*})$ a weak-mild solution to equation \eqref{kinetic-fokker-planck} if for every $f\in \mathcal{H}^s_{\kin} (\bR^{2d})$ and $t \in [0,T]$, it holds
\begin{equation}
\label{eq:weak-mild-solution}
    \langle \nu_t, h\rangle_{-s, s}=\langle \nu_0, P_t f\rangle_{-s, s}+\int_0^t\langle  \nu_r, (\nabla_v P_{t-r}f) \cdot (\Gamma*\nu_r)\rangle_{-s, s} \dd r\,,
\end{equation}
where $(\Gamma*\nu_r)$ is the generalised convolution $
(\Gamma*\nu)(x,v)= \langle \nu ,\Gamma((x,v), (\cdot,\cdot)) \rangle_{-s,s}$.
\end{definition}

Uniqueness can be readily established by using a classical argument. 

\begin{proposition}[Uniqueness]
\label{pro:uniqueness}
    Suppose that $\Gamma \colon \bR^{2d}\times \bR^{2d}\to \bR^{d}$  is a function  that satisfies the following property 
    \begin{equation}
\label{eq:property_Gamma}\Vert\Gamma\Vert_{s,\alpha}= \Vert \Vert \Gamma((x,v), (y,w))\Vert_{ \mathcal{H}^{s}_{\kin}(y,w)}\Vert_{\mathcal{C}^{\alpha}_{\kin}(x,v)} <\infty
\end{equation}
   for some $s>0$ and $\alpha>s$ non-integer. Then, there exists a unique weak-mild solution to \eqref{kinetic-fokker-planck}. 
\end{proposition}

\begin{proof}
Suppose $\nu, \rho\in L^\infty([0,T], (\mathcal{H}^{s}_{\kin} (\bR^{2d}))^{*})$ are two weak-mild solutions. Take the difference between the two equations \eqref{d:weak-mild-solution} and obtain that for every $f\in \mathcal{H}^s_{\kin} (\bR^{2d})$
\begin{equation*}
\begin{split}
\langle \nu_t-\rho_t, f\rangle_{-s, s}=& \int_0^t \langle \nu_r-\rho_r, (\nabla_v P_{t-r}f)\cdot(\Gamma*\nu_r)\rangle_{-s, s} \dd r \\
&+\int_0^t \langle \rho_r, (\nabla_v P_{t-r}h)\cdot(\Gamma*(\nu_r-\rho_r))\rangle_{-s, s} \dd r.
\end{split}
\end{equation*}
In particular,
\begin{equation*}
\begin{split}
\norm{\nu_t - \rho_t}_{-s} \leq & \int_0^t \norm{\nu_r-\rho_r}_{-s}\sup_{\norm{f}_s\leq 1} \norm{(\nabla_v P_{t-r}f)\cdot(\Gamma*\nu_r)}_s \dd r \\
& + \int_0^t \norm{\rho_r}_{-s} \sup_{\norm{f}_s\leq 1}\norm{(\nabla_v P_{t-r}f)\cdot(\Gamma*(\nu_r-\rho_r))}_s \dd r.
\end{split}
\end{equation*}
We use the norm inequality for products obtained in Proposition \ref{pro:product-besov} as well as the regularization properties of the kinetic semigroup, see Lemma \ref{lem:semigroup-time}. Therefore for any $\mu \in (\mathcal{H}^{s}_{\kin} (\bR^{2d}))^{*}$  and $f\in \mathcal{H}^{s}_{\kin} (\bR^{2d})$  with $\norm{f}_s\leq 1$ there exists two constants $C,C'>0$ such that 
\begin{equation*}
\begin{split}
\norm{(\nabla_v P_{t-r}f)\cdot(\Gamma*\mu)}_s &\leq C\norm{\nabla_v P_{t-r}f}_s \norm{\Gamma*\mu}_{\mathcal{C}^{\alpha}_{\kin}} \leq \frac{C'}{\sqrt{t-r}} \norm{\Gamma*\mu}_{\mathcal{C}^{\alpha}_{\kin}}.
\end{split}
\end{equation*}
Moreover by definition of generalised convolution  together with \eqref{CS_sobolev} one has
\[
\norm{\Gamma*\mu}_{\mathcal{C}^{\alpha}_{\kin}}\leq\norm{\mu}_{-s}\norm{\Gamma}_{s,\alpha}\,,\]
Combining these two estimates, we obtain
\begin{equation*}
\norm{(\nabla_v P_{t-r}f)\cdot(\Gamma*\mu)}_s  \leq \frac{C'}{\sqrt{t-r}} \norm{\mu}_{-s}\norm{\Gamma}_{s,\alpha}
\end{equation*}
Using now the hypothesis of $\nu \in L^\infty([0,T], (\mathcal{H}^{s}_{\kin} (\bR^{2d}))^{*})$, we conclude that there exists a (new) constant $M>0$ such that
\begin{equation*}
\norm{\nu_t - \rho_t}_{-s} \leq M \norm{\Gamma}_{s,\alpha} \int_0^t \frac 1 {\sqrt{t-r}}\norm{\nu_r-\rho_r}_{-s} \dd r.
\end{equation*}
From Gronwall lemma we obtain the proof.
\end{proof}
Using the properties of the embeddings in Proposition \ref{pro:nabla-f} the condition $\Vert\Gamma\Vert_{s,\alpha}<\infty$ implies that $\Gamma$ is bounded and Lipschitz for any $s>2d +3$ and $\alpha> s$ non-integer, therefore  we can relate classical weak solutions $\mu\in C([0, T],\cP(\mathbb{R}^{2d}))$   from  Proposition \ref{pro:known-weak-solution} with weak-mild solution \eqref{eq:weak-mild-solution}.
\begin{proposition}\label{equivalence_solutions}
   Let $\nu_0\in \mathcal{P}(\bR^{2d})$ and $\Gamma$ such that  $\Vert\Gamma\Vert_{s,\alpha}<\infty$ for some $\alpha>s$ non-integer and $s\geq 2d+3$. Then the weak solution $\nu$ of \eqref{kinetic-fokker-planck} is the unique weak-mild solution.
\end{proposition}
\begin{proof}
    The proof simply follows by checking that one can pass from equation \eqref{weak_solution} to \eqref{eq:weak-mild-solution} For any fixed $t \in [0,T]$ and any smooth compactly supported function  $f\colon \bR^{2d}\to \bR$ we  consider the function $\psi(r,x,v) =P_{t-r}f(x,v)$  and we want to evaluate \[\langle \nu_r,\psi(r, \cdot, \cdot)\rangle=\langle \nu_r,\psi(r, \cdot, \cdot)\rangle_{-s,s}\]
    on the values $0$ and $t$ for any weak solution $\nu\in C([0, T],\cP(\mathbb{R}^{2d}))$. By simply applying the chain rule and \eqref{eq:kinetic-semigroup-ito}, we obtain
    \[\langle\nu_t,f\rangle= \langle\nu_0,P_{t}f\rangle +\int_0^t \langle  \nu_r , (\Gamma*\nu_s)\cdot \nabla_vP_{t-r}f\rangle \dd r\,.\]
    Then the general result holds by density of  smooth compactly supported function  in $\mathcal{H}^{s}_{\kin} (\bR^{2d})$. From the uniqueness of weak-mild solutions in Proposition \ref{pro:uniqueness} we obtain the thesis.
\end{proof}

\section{Law of large numbers}
\label{s:lln}
We study the properties of the empirical measure of system \eqref{eq:kinetic-diffusions}.
\begin{proposition}
    \label{p:mild-formulation}
    Let $s>2d + 3$. The empirical measure $\nu^N $ of the system  \eqref{eq:kinetic-diffusions} satisfies for any $t \in [0,T]$ and  every $f \in \mathcal{H}^s_{\kin} (\bR^{2d})$ the following a.s. identity 
    \begin{equation}
    \label{eq:kinetic-mild-emp-mes}
        \langle \nu^N_t, f\rangle_{-s, s} = \langle \nu^N_0, P_t f\rangle_{-s, s} + z^N_t(f)+ \int_0^t\langle  \nu^N_r, \nabla_v (P_{t-r}f) \cdot\left( \Gamma*\nu^N_r\right)\rangle_{-s, s} \dd r,
    \end{equation}
    where the random distribution $z^N_t$ is defined for any $f\in \mathcal{H}^s_{\kin} (\bR^{2d})$ by
    \begin{equation}
    \label{d:kinetic-z-n-t}
        z^N_t(f) = \frac {\sqrt{2}}{ N} \sum_{i=1}^N \int_0^t \nabla_v ( P_{t-r} f)(x^{i,N}_r, v^{i,N}_r) \cdot \dd B^i_r.
    \end{equation}
\end{proposition}
\begin{proof}
We will use the same approach as in the proof of Proposition \ref{equivalence_solutions}. Thanks to Proposition \ref{pro:nabla-f}, we have that $f\in \mathcal{C}^{3}_{\kin}(\bR^{2d})$, therefore for any fixed $t \in [0,T]$ we can  apply It\^o formula to the function $\psi(r,x,v) =P_{t-r}f(x,v)$ to each component of \eqref{eq:kinetic-diffusions} to obtain for any $i=1,\dots, N$ the a.s. identity
\begin{equation*}
\begin{split}
    &f(x^{i,N}_t,v^{i,N}_t)= P_{t}h(x^{i,N}_0,v^{i,N}_0)+\sqrt{2}\int_0^t \nabla_v (P_{t-r} f)(x^{i,N}_r, v^{i,N}_r) \cdot \dd B^i_r \\& + \int_0^t \nabla_v (P_{t-r} f)(x^{i,N}_r, v^{i,N}_r) \cdot\left( \frac{1}{N}\sum_{j\neq i} \Gamma((x^{i,N}_r,v^{i,N}_r),(x^{j,N}_r,v^{j,N}_r))\right)\dd r\,\,,
\end{split}
\end{equation*}
where the missing terms are compensated by the PDE \eqref{kinetic-fokker-planck}, recall Equation \eqref{eq:kinetic-semigroup-ito}. By summing this identity over $i$
and dividing by $N$ and using the definition of the empirical measure \eqref{def:emp-measure-kinetic-diffusions}, we derive \eqref{eq:kinetic-mild-emp-mes}.
\end{proof}

\begin{proposition}
\label{p:fourier-z}
  The stochastic term $z^N_t(f)$ defined in \eqref{d:kinetic-z-n-t} can be rewritten for any $f \in \mathcal{H}^s_{\kin} (\bR^{2d})$ as
   \begin{equation}
    \label{z:fourier}
    z^N_t(f)= \frac{i\sqrt{2}}{(2\pi)^{2d}}\frac 1N \sum_{i=1}^N \int_{\bR^{2d}} \int_0^t  \exp\left(i (\xi \cdot x^{i,N}_r + \eta \cdot v^{i,N}_r)\right) G(t-r, \xi, \eta ) \, \hat{f}(\xi, \eta) \eta  \,\cdot \dd B^i_r\, \dd \xi \dd \eta,
\end{equation}
where $\hat{f}(\xi, \eta) = \cF(f)(\xi, \eta)$ and $G: [0, T] \times \bR^{2d} \to (0, \infty)$ is defined by equation \eqref{def:kinetic-kernel}.
\end{proposition}

\begin{proof}
Using the Fourier representation given in Lemma \ref{lem:semigroup-fourier} and the fact that $\nabla_v$ transforms into a multiplication for $\eta$, we obtain that
\begin{equation*}
\nabla_{v} (P_{t-r} f)(x, v)= i\int_{\bR^{2d}} \exp\left(i (\xi \cdot x + \eta \cdot v)\right) G(t-r, \xi, \eta ) \, \hat{f}(\xi, \eta) \eta \, \dd \xi \dd \eta
\end{equation*}
Plugging this expression into equation \eqref{d:kinetic-z-n-t} yields that
\begin{equation*}
    z^N_t(f) = \frac{i \sqrt{2}}N \sum_{i=1}^N \int_0^t \left( \int_{\bR^{2d}} \exp\left(i (\xi \cdot x^{i,N}_r + \eta \cdot v^{i,N}_r)\right) G(t-r, \xi, \eta ) \, \hat{f}(\xi, \eta) \eta \, \dd \xi \dd \eta \right) \cdot \dd B^i_r
\end{equation*}
As the process $i(\xi \cdot x^{i,N}_r + \eta \cdot v^{i,N}_r)$ is measurable (and its exponential uniformly bounded), we can apply the standard stochastic Fubini theorem  and conclude the proof.
\end{proof}

To give a clear estimate of the stochastic term $z^N_t(f)$ we need an elementary lemma.

\begin{lemma}
\label{lem:fourier-integral}
For every $s>2d+2$, one has
\begin{equation*}
\int_{\bR^{2d}} \frac{|\eta|^2 + |\eta|^4}{(1 + |\xi|^{2/3} + |\eta|^2)^s} \, d\xi \, d\eta <+\infty\,.
\end{equation*}
\end{lemma}
\begin{proof}
Fix \( \eta \in \mathbb{R}^d \) and let \( A = 1 + |\eta|^2 \), we first estimate:
\begin{equation*}
\int_{\mathbb{R}^d} \frac{1}{(A + |\xi|^{2/3})^s} \, d\xi.
\end{equation*}
Switching to polar coordinates and substituting \( u = r^{2/3} \), we obtain that
\begin{equation*}
\int_{\mathbb{R}^d} \frac{1}{(A + |\xi|^{2/3})^s} \, d\xi 
= C_d \int_0^\infty \frac{r^{d-1}}{(A + r^{2/3})^s} \, dr 
= \frac {2C_d}3  \int_0^\infty \frac{u^{\frac{3d}{2} - 1}}{(A + u)^s} \, du.
\end{equation*}
With the change of variable $ u = At$, one obtains that the last integral is proportional to \( A^{-s + \frac{3d}{2}} \) times the Beta function $B(\frac{3d}2, s-\frac {3d}2)$ which is finite. So, there exists $C$ independent of $\xi$ such that
\begin{equation*}
\int_{\mathbb{R}^d} \frac{1}{(1 + |\xi|^{2/3} + |\eta|^2)^s} \, d\xi 
\leq C (1 + |\eta|^2)^{-s + \frac{3d}{2}}.
\end{equation*}
We now estimate the integral over \( \eta \)
\begin{equation*}
\int_{\mathbb{R}^d} (1 + |\eta|^2)^{-s + \frac{3d}{2}} (|\eta|^2 + |\eta|^4) \, d\eta.
\end{equation*}
Near \( \eta = 0 \), the integrand is smooth and integrable for any \( s > 0 \). For large \( |\eta| \), the integrand behaves like:
\begin{equation*}
|\eta|^{-2s + 3d + 2} \quad \text{and} \quad |\eta|^{-2s + 3d + 4}.
\end{equation*}
These are integrable at infinity if and only if the exponents are less than \( -d \). That gives $s > 2d + 2$ and the proof is concluded.
\end{proof}
Thanks to these estimate we achieve a full  control of the stochastic term $z^N_t(f)$.
\begin{lemma}
\label{lem:pathwise_bound}
    Assume $s> 2d + 2$. The random perturbation  $z^N$ in equation \eqref{d:kinetic-z-n-t} is a random variable with values in $L^\infty([0,T], (\mathcal{H}^{s}_{\kin} (\bR^{2d}))^{*}$. Moreover, for any $\zeta >0 $ there exists a positive random variable  $C_{\zeta,T}$ depending only on $T$, with at least finite second moment, such that one has the a.s. inequality
    \begin{equation}
         \sup_{t \in [0,T]} \norm{z^N_t}_{-s}  \leq  \frac{C_{\zeta,T}}{N^{1/2-\zeta}} \,.
    \end{equation}
\end{lemma}

\begin{proof}
Let $w^{\zeta,N}_t(h)=N^{1/2-\zeta}z^N_t(h)$. The strategy of the proof is to apply Lemma \ref{lem:grr} after a careful control of the moments of $\Vert w^{\zeta,N}_t- w^{\zeta,N}_u\Vert_{-s}$. In what follows we will use the notations $\lesssim_T$, $\lesssim_m$  and $\lesssim_{m,T}$ to denote inequality up to a constant that depends on $T$, $m$ or $T$ and $m$ respectively. Using \eqref{z:fourier} we write for any 
$ h\in \mathcal{H}^{s}$ 
\begin{align}
 w^{\zeta,N}_t(h)- w^{\zeta,N}_u(h)= \frac{i\sqrt{2}}{(2\pi)^{2d}} \int_{\bR^{2d}} (K_{t, u, \xi,\eta}+H_{t, u, \xi,\eta})   \hat{h}(\xi, \eta)\dd \xi \dd \eta\,,  
\end{align}
where $K_{t, u, \xi,\eta}$ and $H_{t, u, \xi,\eta}$ are the random variables
\begin{equation*}
\begin{split}
     K_{t, u, \xi,\eta}&=\sum_{i=1}^N\int_u^t\frac{1}{N^{1/2+\zeta}}  \exp{(i\xi \cdot x^{i,N}_r+i\eta \cdot v^{i,N}_r)}G(t-r, \xi, \eta) \eta\cdot\dd B^i_r\,,\\ H_{t, u, \xi,\eta}&=\sum_{i=1}^N\int_0^u\frac{1}{N^{1/2+\zeta}}  \exp{(i\xi \cdot x^{i,N}_r+i\eta \cdot v^{i,N}_r)}(G(t-r, \xi, \eta)- G(u-r, \xi, \eta))\eta\cdot\dd B^i_r.
\end{split}
\end{equation*}
We proceed to estimate $\mathbb{E}\Vert w^{\zeta,N}_t- w^{\zeta,N}_u\Vert_{-s}^{2}$. We multiply and divide by the Fourier weight $(1+|\xi|^{2/3}+|\eta|^2 )^{s/2}$ and apply the standard Cauchy-Schwarz inequality,  for any integer $m\geq 1$ there exists a a constant $C_m>0$ such that 
\begin{equation*}
\begin{split}
        &\Vert w^{\zeta,N}_t- w^{\zeta,N}_u\Vert_{-s}^{2m}= \sup_{h\in \mathcal{H}^{s} \colon \norm{h}_{s}\leq 1}|w^{\zeta,N}_t(h)- w^{\zeta,N}_u(h)|^{2m}\\
        &\leq C_m\Bigg( \left(\int_{\bR^{2d}}(1+|\xi|^{2/3}+|\eta|^2 )^{-s}\left\vert  K_{t, u,\xi,\eta}\right\vert^2 \dd \xi \dd \eta \right)^m\\&+ \left(\int_{\bR^{2d}}(1+|\xi|^{2/3}+|\eta|^2 )^{-s}\left\vert  H_{t, u,\xi,\eta}\right\vert^2 \dd \xi \dd \eta \right)^m\Bigg),
\end{split}
\end{equation*}
where we have used the classical inequality $(a+b)^{2m}\leq 2m(a^{2m}+ b^{2m})$.
Using the It\^o isometry and the estimates on $G$ given in Lemma \ref{lem:kernel-time-regularity}, we obtain that
\begin{equation*}
\begin{split}
&\bE\Vert w^{\zeta,N}_t- w^{\zeta,N}_u\Vert_{-s}^{2}\leq C_1 N^{2\zeta} \bigg(\int_{\bR^{2d}}(1+|\xi|^{2/3}+|\eta|^2 )^{-s}\int_u^t  G^2(t-r, \xi, \eta) |\eta|^2\dd r \dd \xi \dd \eta\\&+ \int_{\bR^{2d}} (1+|\xi|^{2/3}+|\eta|^2 )^{-s}\int_0^u  |G(t-r, \xi, \eta) -G(u-r,\xi, \eta)|^2 |\eta|^2\dd r\dd \xi \dd \eta\bigg)\\&\lesssim_T  N^{2\zeta} |t-u| \int_{\bR^{2d}}(1+|\xi|^{2/3}+|\eta|^2 )^{-s}(|\eta|^2+  |\eta|^4)\dd \xi \dd \eta\,.
\end{split}
\end{equation*}

Writing the powers of integrals as integrals over extended variables for any integer $m>1$ one has 
\begin{equation*}
\begin{split}
&\left(\int_{\bR^{2d}}(1+|\xi|^{2/3}+|\eta|^2 )^{-s}\left\vert  K_{t, u,\xi,\eta}\right\vert^2 \dd \xi \dd \eta \right)^m= \\
\int_{(\bR^{2d})^m}&(1+|\xi_1|^{2/3}+|\eta_1|^2 )^{-s}\ldots(1+|\xi_m|^{2/3}+|\eta_m|^2 )^{-s}\\&\left\vert  K_{t, u,\xi_1,\eta_1}\right\vert^2 \ldots\left\vert  K_{t, u,\xi_m,\eta_m}\right\vert^2\dd \xi_1 \ldots  \dd \eta_m.
\end{split}
\end{equation*}
Taking the expectation and using Fubini, all terms that include an even number of Brownian motion vanish. As in the proof of \cite[Lemma 3.1]{bertini_synchronization_2014}, the number of non-zero terms can be bounded by a (new) constant $C_m$ depending only on $m$ and by the estimate with $m=1$, i.e., by
\begin{equation*}
\begin{split}
&\frac{C_m}{N^{2m\zeta}}\left(\int_{\bR^{2d}}(1+|\xi|^{2/3}+|\eta|^2 )^{-s}\int_u^t  G^2(t-r, \xi, \eta) |\eta|^2\dd r \dd \xi \dd \eta\right)^m\\
& \leq \frac{C_m}{N^{2m\zeta}}\left(\int_{\bR^{2d}}(1+|\xi|^{2/3}+|\eta|^2 )^{-s} |\eta|^2\dd \xi \dd \eta\right)^m|t-u|^m,
\end{split}
\end{equation*}
where in the second step we have used Lemma \ref{lem:kernel-time-regularity}. Similarly
\begin{equation*}
\begin{split}
&\bE\left(\int_{\bR^{2d}}(1+|\xi|^{2/3}+|\eta|^2 )^{-s}\left\vert  H_{t, u,\xi,\eta}\right\vert^2 \dd \xi \dd \eta \right)^m \\& \lesssim_{m}\frac{1}{N^{2m\zeta}}\left(\int_{\bR^{2d}}(1+|\xi|^{2/3}+|\eta|^2 )^{-s}\int_0^u  |G(t-r, \xi, \eta) -G(u-r,\xi, \eta)|^2 |\eta|^2\dd r \dd \xi \dd \eta\right)^m\\& \lesssim_T \frac{1}{N^{2m\zeta}}\left(\int_{\bR^{2d}}(1+|\xi|^{2/3}+|\eta|^2 )^{-s} |\eta|^4\dd \xi \dd \eta\right)^m|t-u|^m
\end{split}
\end{equation*}
which is again finite. Hence, we obtain 
\begin{equation*}
\bE\Vert w^{\zeta,N}_t- w^{\zeta,N}_u\Vert_{-s}^{2m}\lesssim_{m,T}\frac{1}{N^{2m\zeta}}|t-u|^m.
\end{equation*}
We are now able to apply Lemma \ref{lem:grr} with  $\Psi(u)=u^{2m}$ and $p(u)= u^{\frac{2+\gamma}{2m} }$ (with $0<\gamma<1$) and obtain
\begin{equation}\label{estimate_GRR}
    \begin{split}
         \sup_{t \in [0, T]}\norm{w^{\zeta,N}_t}_{-s} &\lesssim_{m,\gamma,T} \left( \int_{[0, T]^2} \frac{\Vert w^{\zeta,N}_t- w^{\zeta,N}_u\Vert_{-s}^{2m}}{| t - u
   |^{2+\gamma}} \dd u \dd t  \right)^{1 / 2m}=Y_N
    \end{split}
\end{equation}
together with
\begin{equation*}\mathbb{E} [Y_N^{2m}]\lesssim C_T\frac{1}{N^{2m\zeta}}.
\end{equation*}
In particular,
\begin{equation*}
\mathbb{P}\left(\sup_N Y_N\geq \lambda\right) \leq \sum_N \frac{\mathbb{E}[Y^{2m}]}{\lambda^{2m}} \leq \frac {C_T}{\lambda^{2m}} \sum_{N} \frac{1}{N^{2m\zeta}} \lesssim_{T,\zeta} \frac{1}{\lambda^{2m}},
\end{equation*}
where the series $\sum_{N}N^{-2m\zeta}$ converges as soon as $m>(2\zeta)^{-1}$. Consequently
\begin{equation}
    \mathbb{E} \left[(\sup_N Y_N)^{2}\right]=2\int_0^\infty \lambda\, \mathbb{P}(\sup_N Y_N>\lambda)\, \dd \lambda
    \lesssim_{T,\zeta,m }\int_0^\infty \lambda^{1-2m} \dd \lambda,
\end{equation}
and the last integral is finite precisely when $m>1$. Therefore taking $m>\max((2\zeta)^{-1},1)$, we get $\mathbb{E}[(\sup_N Y_N)^2]<\infty$. Denoting $\sup_N Y_N=C_{\zeta,T}$, we are done.
\end{proof}

We are now ready for our main result.

\begin{theorem}
\label{thm:lln}
Assume $s>2d + 3$ and that $\Vert\Gamma\Vert_{s,\alpha}<\infty$ for some $\alpha>s$, non-integer. Let $(\nu_t)_{t\in[0,T]}$ be a weak-mild solution to Equation \eqref{kinetic-fokker-planck}. For every $N=2, 3, \dots$, the empirical measure $\nu^N = (\nu^N_t)_{t\in [0,T]}\in L^{\infty}([0,T],(\mathcal{H}^{s}_\kin(\bR^{2d}))^*)$ a.s. and for any $\zeta > 0$, there exists a constant $C_{\Gamma, T, \zeta}>0$ only depending on $\Gamma, T$ and $\zeta$ such that
\begin{equation*}
    \mathbb{E}\left[ \sup_{t\in [0,T] }\norm{\nu^N_t -\nu_t}_{-s} \right] \leq C_{\Gamma, T, \zeta} \left(\mathbb{E}\left[\norm{\nu^N_0 - \nu_0}_{-s}\right] + \frac {1}{N^{1/2-\zeta}} \right).
\end{equation*}
Moreover, if $\nu_0\in\mathcal{P}(\bR^{2d})$ and $\lim_{N\to \infty} \norm{\nu^N_0 - \nu_0}_{-s} =0$ in probability, then $\nu^N$ converges to $\nu\in C([0,T], \mathcal{P}(\bR^{2d}))$ in probability.
\end{theorem}

\begin{proof}
The proof mimics the one of Proposition \ref{pro:uniqueness} using the fact that $\nu^N$ satisfies a weak-mild formulation (Proposition \ref{p:mild-formulation}) which is very similar to the mild PDE satisfied by $\nu$, and the fact that $(\mathcal{H}^{s}_\kin(\bR^{2d}))^*$ is a Hilbert space. Take the difference between these two equations and write it as
\begin{equation*}
\begin{split}
\langle \nu^N_t-\nu_t, f\rangle_{-s, s}=& \langle \nu^N_0 - \nu_0, P_tf\rangle_{-s,s} + \int_0^t \langle \nu^N_r-\nu_r, (\nabla_v P_{t-r}f)[\Gamma*\nu^N_r]\rangle_{-s, s} \dd r \\
&+\int_0^t \langle \nu_r, (\nabla_v P_{t-r}f)[\Gamma*(\nu^N_r-\nu_r)]\rangle_{-s, s} \dd r + z^N_t(f).
\end{split}
\end{equation*}
Take the supremum with respect to $f\in \mathcal{H}^s_{\kin}(\bR^{2d})$ such that $\norm{f}_s=1$, use the triangular inequality as well as the hypothesis on $\Gamma$ and the properties of the semigroup, one obtains for some constant $C>0$
\begin{equation*}
    \norm{\nu^N_t - \nu_t}_{-s} \leq \frac{C}{\sqrt{t}} \norm{\nu^N_0 - \nu_0}_{-s} + \int_0^t \frac {C}{\sqrt{t-r}} \norm{\nu^N_r -\nu_r}_{-s} \dd r + \sup_{r\in [0,T]} \norm{z^N_r}_{-s}.
\end{equation*}
Gronwall inequality and Lemma \ref{lem:pathwise_bound} yield the first part of the proof.   The second part is indeed consequence of Proposition \eqref{equivalence_solutions}.
\end{proof}

\bibliographystyle{alpha}
\bibliography{bibliov2}

\end{document}